\documentclass[reqno]{amsart}
\usepackage{txfonts}
\usepackage{amssymb}
\usepackage{amsthm}
\usepackage{amsmath}
\usepackage{mathrsfs}

\usepackage[utf8]{inputenc}

\theoremstyle{plain}
\newtheorem{theorem}{Theorem}[section]
\newtheorem{lemma}[theorem]{Lemma}
\newtheorem{corollary}[theorem]{Corollary}
\newtheorem{proposition}[theorem]{Proposition}

\newtheorem{thm}{Theorem}

\theoremstyle{definition}
\newtheorem{definition}[theorem]{Definition}

\theoremstyle{remark}
\newtheorem{remark}{Remark}

\newcommand{\C}{\mathbb{C}}
\newcommand{\N}{\mathbb{N}}
\newcommand{\R}{\mathbb{R}}

\newcommand{\mcC}{\mathcal{C}}
\newcommand{\mcE}{\mathcal{E}}
\newcommand{\mcH}{\mathcal{H}}
\newcommand{\mcL}{\mathcal{L}}
\newcommand{\mcM}{\mathcal{M}}

\newcommand{\mrmT}{\mathrm{T}}
\newcommand{\inv}{\mathrm{inv}}

\newcommand{\TMS}{\Sigma_{{\bf A}}^+}
\newcommand{\tms}{\sigma_{{\bf A}}}
\newcommand{\bTMS}{\Sigma_{{\bf B}}^+}
\newcommand{\btms}{\sigma_{{\bf B}}}

\DeclareMathOperator{\Aut}{Aut}

\title[\null]{Weak rigidity of entropy spectra}

\author[\null]{Katsukuni\ \ Nakagawa}

\address{K.\ Nakagawa\\
Graduate School of Advanced Science and Engineering\\
Hiroshima University\\
Higashi-Hiroshima 739-8526\\
Japan}
\email{ktnakagawa@hiroshima-u.ac.jp}

\begin{document}



\begin{abstract}
In this paper, we consider entropy spectra on topological Markov shifts.
We prove that if two measure-preserving dynamical systems of Gibbs measures with H\"{o}lder continuous potentials are isomorphic, then their entropy spectra are the same.
This result raises a new rigidity problem.
We call this problem the weak rigidity problem, contrasting it with the strong rigidity problem proposed by Barreira and Saraiva.
We give a complete answer to the weak rigidity problem for Markov measures on a topological Markov shift with an aperiodic transition matrix of size 2.
\end{abstract}

\maketitle


\section{Introduction}

\label{sec:introduction}

Let $N\ge2$ be an integer and ${\bf A}$ an $N\times N$ zero-one matrix.
We set
\[
\TMS=\{\omega=(\omega_n)_{n\in\N\cup\{0\}}\in\{1,\dots,N\}^{\N\cup\{0\}}:{\bf A}(\omega_n\omega_{n+1})=1,\ n\in\N\cup\{0\}\}
\]
and define the metric $d$ on $\TMS$ by 
\[
d(\omega,\omega^{\prime})=\sum_{n=0}^{\infty}\frac{|\omega_n-\omega^{\prime}_n|}{2^n}.
\]
Then, $\TMS$ is a compact metric space and the \emph{shift map} $\tms:\TMS\to\TMS$ defined by 
\[
(\tms\omega)_n=\omega_{n+1},\quad n\in\N\cup\{0\}
\]
is a continuous map. 
We call the topological dynamical system $(\TMS,\tms)$ a \emph{(one-sided) topological Markov shift} and ${\bf A}$ the \emph{transition matrix} of the shift.
We say that ${\bf A}$ is \emph{aperiodic} if there exists a positive integer $k$ such that all entries of the power ${\bf A}^k$ are positive.
The topological Markov shifts that we consider in this paper are assumed to have aperiodic transition matrices.
Let $n\in\N$. 
A function $\phi:\TMS\to\C$ is said to be $n$-\emph{locally constant} if $\phi(\omega)=\phi(\omega^{\prime})$ for any $\omega,\omega\in\TMS$ with $\omega_k=\omega_k^{\prime},\ 0\le k\le n-1$.
We denote by $L_n(\TMS;\R)$ the set of all real-valued $n$-locally constant functions on $\TMS$.

Let $\mcM_{\inv}(\TMS)$ be the set of all $\tms$-invariant Borel probability measures on $\TMS$.
For $\mu\in\mcM_{\inv}(\TMS)$, we define the entropy spectrum of $\mu$ as follows:

\begin{definition}
\label{def:524}
Let  $\mu\in\mcM_{\inv}(\TMS)$.
We set
\[
E_{\alpha}^{(\mu)}=\left\{\omega\in\TMS:\lim_{n\to\infty}-\frac{1}{n}\log\mu\left(\left\{\omega^{\prime}\in\TMS:\omega_k=\omega^{\prime}_k,\ 0\le k\le n-1\right\}\right)=\alpha\right\}
\]
for $\alpha\in\R$ and define the function $\mcE^{(\mu)}:\R\to\R$ by
\[
\mcE^{(\mu)}(\alpha)=h(\tms|E_{\alpha}^{(\mu)}).
\]
Here, for $Z\subset\TMS$, we denote by $h(\tms|Z)$ the topological entropy of $Z$.
The function $\mcE^{(\mu)}$ is called the \emph{entropy spectrum} of $\mu$.
\end{definition}

The entropy spectrum $\mcE^{(\mu)}$ has a wealth of information about the measure $\mu$.
A \emph{rigidity} is the phenomenon that a measure is `recovered' from its entropy spectrum.
In this paper, we consider the rigidities of entropy spectra of Gibbs measures for H\"{o}lder continuous functions (see Section \ref{sec:preliminaries} for the definition of a Gibbs measure).

For a H\"{o}lder continuous function $f:\TMS\to\R$, we denote by $\mu_f$ the (unique) Gibbs measure for $f$. 
For measure-preserving dynamical systems of Gibbs measures, we give the following two definitions:

\begin{definition}
\label{def:172}
Let $f,g:\TMS\to\R$ be H\"{o}lder continuous.
The two measure-preserving dynamical systems $(\TMS,\tms,\mu_f)$ and $(\TMS,\tms,\mu_g)$ are said to be \emph{equivalent} if there exists a homeomorphism $\tau:\TMS\to\TMS$ such that $\tau\circ\tms=\tms\circ\tau$ and $\mu_f=\mu_g\circ\tau$.
\end{definition}

\begin{definition}
\label{def:019}
Let $f,g:\TMS\to\R$ be H\"{o}lder continuous.
The two measure-preserving dynamical systems $(\TMS,\tms,\mu_f)$ and $(\TMS,\tms,\mu_g)$ are said to be \emph{isomorphic} if there exist two Borel sets $X, Y\subset\TMS$ and a bijection $\Phi:X\to Y$ such that the following four conditions hold:
\begin{itemize}
\item$\tms(X)\subset X,\ \tms(Y)\subset Y$ and $\mu_f(X)=\mu_g(Y)=1$.
\item Both $\Phi:X\to Y$ and $\Phi^{-1}:Y\to X$ are Borel measurable. 
\item$\Phi\circ\tms=\tms\circ\Phi\ \mbox{on}\ X$.
\item$\mu_f=\mu_g\circ \Phi$.
\end{itemize}
\end{definition}

Let $\mcC$ be a set consisting of real-valued H\"{o}lder continuous functions on $\TMS$.
For $f\in\mcC$, we define the following three equivalence classes of $f$:
\begin{align*}
&[f]_1=\{g\in\mcC:\mcE^{(\mu_f)}=\mcE^{(\mu_g)}\},\\
&[f]_2=\{g\in\mcC:(\TMS,\tms,\mu_f)\ \mbox{and}\ (\TMS,\tms,\mu_g)\ \mbox{are equivalent}\},\\
&[f]_3=\{g\in\mcC:(\TMS,\tms,\mu_f)\ \mbox{and}\ (\TMS,\tms,\mu_g)\ \mbox{are isomorphic}\}.
\end{align*}

It is obvious that if $(\TMS,\tms,\mu_f)$ and $(\TMS,\tms,\mu_g)$ are equivalent, then they are isomorphic; we have $[f]_2\subset[f]_3$ for $f\in\mcC$.
In \cite{Barreira}, Barreira and Saraiva proved the following theorem:

\begin{thm}[\text{\cite[Proposition 2]{Barreira}}]
\label{thm:274}
Let $f,g:\TMS\to\R$ be H\"{o}lder continuous.
If $(\TMS,\tms,\mu_f)$ and $(\TMS,\tms,\mu_g)$ are equivalent, then $\mcE^{(\mu_f)}=\mcE^{(\mu_g)}$.
\end{thm}

Theorem \ref{thm:274} implies that $[f]_2\subset[f]_1$ for $f\in\mcC$.
Therefore, we have the following rigidity problem:
\begin{center}
\begin{minipage}{0.9\textwidth}
($\mathrm{SR}$).\ 
Identify the set $\{f\in\mcC:[f]_2=[f]_1\}$.
\end{minipage}
\end{center}
We call the problem ($\mathrm{SR}$) the \emph{strong rigidity problem} for $\mcC$.

In this paper, we prove the following extension of Theorem \ref{thm:274} implying $[f]_3\subset[f]_1$ for $f\in\mcC$:

\begin{theorem}
\label{thm:401}
Let $f,g:\TMS\to\R$ be H\"{o}lder continuous.
If $(\TMS,\tms,\mu_f)$ and $(\TMS,\tms,\mu_g)$ are isomorphic, then $\mcE^{(\mu_f)}=\mcE^{(\mu_g)}$.
\end{theorem}

Theorem \ref{thm:401} raises the following new rigidity problem:
\begin{center}
\begin{minipage}{0.9\textwidth}
($\mathrm{WR}$).\ 
Identify the set $\{f\in\mcC:[f]_3=[f]_1\}$.
\end{minipage}
\end{center}
We call the problem ($\mathrm{WR}$) the \emph{weak rigidity problem} for $\mcC$.

In \cite{Barreira}, Barreira and Saraiva gave a complete answer to the strong rigidity problem for $\mcC=L_2(\TMS;\R)$ when ${\bf A}$ is a $2\times 2$ zero-one aperiodic matrix.
In this paper, we give a complete answer to the weak rigidity problem for $\mcC=L_2(\TMS;\R)$, in the same set-up.
In fact, we prove that
\begin{equation}
\label{eq:272}
\{f\in\mcC:[f]_3=[f]_1\}=\{f\in\mcC:[f]_2=[f]_1\}
\end{equation}
for $\mcC=L_2(\TMS;\R)$ when ${\bf A}$ is a $2\times 2$ zero-one aperiodic matrix.

This paper is organized as follows.
In Section \ref{sec:preliminaries}, we give preliminary definitions and basic facts. 
Especially, the construction of Gibbs measures via transfer operators is reviewed.
In Section \ref{sec:main}, we prove Theorem \ref{thm:401}.
Moreover, we prove that (\ref{eq:272}) holds for $\mcC=L_2(\TMS;\R)$ when ${\bf A}$ is a $2\times 2$ zero-one aperiodic matrix; see Corollaries \ref{cor:732} and \ref{cor:558} below.
In Appendix \ref{app:non-rigidity}, we study the set $\{f\in\mcC:[f]_2=[f]_3\}$ for $\mcC=L_2(\TMS;\R)$.
We give a sufficient condition of ${\bf A}$ for which this set contains an open and dense subset of $L_2(\TMS;\R)$.
This yields an extension of a `non-rigidity' result of \cite{Barreira}; see Corollary \ref{cor:778} below.


\section{Preliminaries}
\label{sec:preliminaries}

An element of $\bigcup_{n\in\N\cup\{0\}}\{1,\dots,N\}^n$ is called a \emph{word}. 
For $n\in\N\cup\{0\}$ and a word $w\in\{1,\dots,N\}^n$, we write $|w|=n$.
Moreover, we write $w=w_0\cdots w_{|w|-1}$ for a word $w$, where $w_k\in\{1,...,N\},\ 0\le k\le|w|-1$. 
The \emph{empty word} is the unique word $w$ with $|w|=0$. 
A word $w$ is said to be ${\bf A}$-\emph{admissible} if $|w|\le1$ or if $|w|\ge2$ and ${\bf A}(w_{k}w_{k+1})=1$ for $0\le k\le|w|-1$. 
For $n\in\N\cup\{0\}$, we denote by $W_{{\bf A}}^n$ the set of all ${\bf A}$-admissible words $w$ with $|w|=n$.
For $\omega\in\TMS$ and $n\in\N\cup\{0\}$, we define the ${\bf A}$-admissible word $\omega|n$ with $|\omega|n|=n$ by
\[
\omega|n=\begin{cases}
\omega_{0}\cdots\omega_{n-1}&(n\ge1),\\
\mbox{the empty word}&(n=0).
\end{cases}
\]
Moreover, for $n\in\N\cup\{0\}$ and a word $w$ with $|w|=n$, we set
\[
[w]=\{\omega\in\TMS:\omega|n=w\}.
\]

For $\phi:\TMS\to\C$ and $n\in\N$, we write $S_n\phi=\sum_{k=0}^{n-1}\phi\circ\tms^k$.
We give the definition of a Gibbs measure as follows:

\begin{definition}
\label{def:042}
Let $\mu\in\mcM_{\inv}(\TMS)$ and let $f:\TMS\to\R$ be continuous.
We call $\mu$ a \emph{Gibbs measure} for $f$ if there exist $C>0$ and $P\in\R$ such that the following inequality holds for $\omega\in\TMS$ and $n\in\N$:
\[
C^{-1}\le\frac{\mu([\omega|n])}{\exp(-nP+S_nf(\omega))}\le C.
\]
\end{definition}
\begin{remark}
\label{rem:785}
Let $\mu$ be a Gibbs measure.
By Definition \ref{def:042}, we have $\mu(G)>0$ for any non-empty open subset $G$ of $\TMS$.
\end{remark}

We construct Gibbs measures for real-valued H\"{o}lder continuous functions, using transfer operators.
First, we define
\[
P(\tms,f)=\sup_{\mu\in\mcM_{\inv}(\tms)}\left(h_{\mu}(\tms)+\int f\,d\mu\right)
\]
for a continuous function $f:\TMS\to\R$.
Here, $h_{\mu}(\tms)$ denotes the Kolmogorov-Sinai entropy of $\mu\in\mcM_{\inv}(\TMS)$. 
It is easy to see that $P(\tms,f)$ is finite.
We call $P(\tms,f)$ the \emph{topological pressure} of $f$.

Next, we define a transfer operator.
For $\alpha>0$, we denote by $\mcH_{\alpha}(\TMS)$ the set of all complex-valued $\alpha$-H\"{o}lder continuous functions on $\TMS$.
We equip $\mcH_{\alpha}(\TMS)$ with the H\"{o}lder norm.
Then, $\mcH_{\alpha}(\TMS)$ is a complex Banach space.
For $f\in\mcH_{\alpha}(\TMS)$,
we define the bounded linear operator $\mcL_f:\mcH_{\alpha}(\TMS)\to \mcH_{\alpha}(\TMS)$ by  
\[
(\mcL_f\phi)(\omega)=\sum_{\omega^{\prime}\in\TMS:\,\tms\omega^{\prime}=\omega}e^{f(\omega^{\prime})}\phi(\omega^{\prime})
\]
and call it the \emph{transfer operator} of $f$.

For a continuous function $f:\TMS\to\R$, we put
\[
\lambda_f=e^{P(\tms,f)}.
\]
Moreover, we denote by $\mcM(\TMS)$ the set of all Borel probability measures on $\TMS$. 
The following theorem is called the Ruelle-Perron-Frobenius theorem (for the proof, see, e.g., \cite[Theorem 2.2]{Parry-Pollicott}).

\begin{theorem}
\label{thm:410}
Let $\alpha>0$ and let $f\in\mcH_{\alpha}(\TMS)$ be real-valued.
\begin{enumerate}
\item[(i)]$\lambda_f$ is an eigenvalue of $\mcL_f:\mcH_{\alpha}(\TMS)\to \mcH_{\alpha}(\TMS)$ with algebraic multiplicity one.
\item[(ii)]There exist $h_f\in\mcH_{\alpha}(\TMS)$ and $\nu_f\in\mcM(\TMS)$ satisfying the following two properties:
\begin{itemize}
\item$h_f>0$ and $\mcL_fh_f=\lambda_fh_f$.
\item$\int\mcL_f\phi\,d\nu_f=\lambda_f\int\phi\,d\nu_f$ for all $\phi\in \mcH_{\alpha}(\TMS)$.
\end{itemize}
\end{enumerate}
\end{theorem}

Using the Ruelle-Perron-Frobenius theorem, we can construct Gibbs measures as follows:

\begin{theorem}
\label{thm:014}
Let $\alpha>0$ and let $f\in \mcH_{\alpha}(\TMS)$ be real-valued.
Let also $h_f$ and $\nu_f$ be as in Theorem \ref{thm:410}, with $\int h_f\,d\nu_f=1$.
Then, the Borel probability measure $\mu$ on $\TMS$ defined by $d\mu=h_f\,d\nu_f$ is a unique Gibbs measure for $f$.
\end{theorem}

For the proof of Theorem \ref{thm:014}, see, e.g., \cite[Corollary 3.2.1]{Parry-Pollicott}.
For a H\"{o}lder continuous function $f:\TMS\to\R$, we denote by $\mu_f$ the unique Gibbs measure for $f$.

For $\nu\in\mcM(\TMS)$, we say that $\nu$ is \emph{non-singular} if $\nu\circ\tms^{-1}$ is absolutely continuous with respect to $\nu$.
Let $\nu\in\mcM(\TMS)$ be non-singular.
For $n\in\N$, we define the finite Borel measure $\nu^{(n)}$ on $\TMS$ by
\[
\nu^{(n)}(B)=\sum_{w\in\{1,\dots,N\}^n}\nu(\tms^n(B\cap[w])).
\]
It is easy to see that $\nu$ is absolutely continuous with respect to $\nu^{(n)}$.
The Radon-Nikodym derivative
\[
J_{\nu}^{(n)}=\frac{d\nu}{d\nu^{(n)}}
\]
is called the $n$th \emph{Jacobian} of $\nu$.
The following lemma plays a key role in the proof of Theorem \ref{thm:401}.

\begin{lemma}
\label{lem:921}
Let $f:\TMS\to\R$ be H\"{o}lder continuous and $\nu_f$ as in Theorem \ref{thm:410}.
Then, $\nu_f$ is non-singular.
Moreover, for $n\in\N$, we have 
\[
J_{\nu_f}^{(n)}=\lambda_f^{-n}\exp(S_nf).
\]
\end{lemma}
The proof of Lemma \ref{lem:921} is easy, so we omit it.

For 2-locally constant functions, we can give a more specific description of their Gibbs measures, using matrix theory.
We denote by $M$ the set of all $N\times N$ non-negative matrices $A$ such that, for any $i,j\in\{1,\dots,N\}$, $A_{ij}>0$ if and only if ${\bf A}_{ij}=1$.
For $a=(a_1,\dots,a_N)\in\R^N$, we denote by $a^{\mrmT}$ the transpose of $a$.

\begin{theorem}
\label{thm:790}
For $A\in M$, the following two assertions hold:
\begin{enumerate}
\item[(i)]There exists a unique positive eigenvalue $\lambda$ of $A$ such that $\lambda>|\eta|$ for any other eigenvalue $\eta\in\C$ of $A$.
Moreover, $\lambda$ has algebraic multiplicity one.
\item[(ii)]Let $\lambda$ be as in (i).
Then, there exist $p,v\in\R^N$ satisfying the following two properties:
\begin{itemize}
\item$p_i>0$ and $v_i>0$ for $i\in\{1,\dots,N\}$.
\item$p$ and $v^{\mrmT}$ are left and right eigenvectors of $\lambda$, respectively.
\end{itemize}
\end{enumerate}
\end{theorem}
For the proof of Theorem \ref{thm:790}, see, e.g., \cite[Theorem 1.1]{Seneta}.
We denote by $\lambda(A)$ the above unique eigenvalue $\lambda$ of $A$ and call it the \emph{Perron root} of $A$.

For $P\in M$, we say that $P$ is a \emph{stochastic matrix} if $\sum_{j=1}^NP_{ij}=1,\ j\in\{1,\dots,N\}$.
Notice that, for any stochastic matrix $P\in M$, the Perron root of $P$ is 1.

\begin{definition}
Let $P\in M$ be a stochastic matrix and $p\in\R^N$ the unique vector  
such that $pP=p$ and $\sum_{i=1}^Np_i=1$.
The \emph{Markov measure} associated with $P$ is the Borel probability measure $\mu$ on $\TMS$ defined by
\[
\mu([w])=p_{w_0}P(w_{0}w_{1})P(w_{1}w_{2})\cdots P(w_{n-2}w_{n-1}),
\]
where $n\in\N\cup\{0\}$ and $w\in\{1,\dots,N\}^n$.
\end{definition}

Let $f\in L_2(\TMS;\R)$.
We define  the matrix $A(f)\in M$ by
\begin{equation}
\label{eq:787}
A(f)_{ij}=\begin{cases}
\exp(f|_{[ij]})&({\bf A}_{ij}=1),\\
0&({\bf A}_{ij}=0).
\end{cases}
\end{equation}
Then, $\lambda_f=\lambda(A(f))$.
Take $v\in\R^N$ so that $v^{\mrmT}$ is a right eigenvector of $\lambda_f$.
We define the probability matrix $P(f)\in M$ by
\[
P(f)_{ij}=\lambda_f^{-1}v_i^{-1}v_jA(f)_{ij}.
\]
The ratio $v_i^{-1}v_j$ is positive and is independent of the choice of $v$.
The following proposition is easy to prove and gives a description of $\mu_f$ as a Markov measure.

\begin{proposition}
Let $f\in L_2(\TMS;\R)$.
Then, the Markov measure associated with $P(f)$ is the unique Gibbs measure for $f$.
\end{proposition}

Let $f:\TMS\to\R$ be H\"{o}lder continuous.
Using topological pressures, we give a description of the entropy spectrum of the Gibbs measure $\mu_f$.
We define the function $\beta:\R\to\R$ by
\[
\beta(q)=P(\tms,qf)-qP(\tms,f).
\]

\begin{theorem}
\label{thm:321}
For a H\"{o}lder continuous function $f:\TMS\to\R$, the following three assertions hold:
\begin{enumerate}
\item[(i)] $\beta$ is strictly decreasing, convex and real analytic. 
\item[(ii)] Let $\alpha_{\min}=-\lim_{q\to\infty}\beta^{\prime}(q)$ and $\alpha_{\max}=-\lim_{q\to-\infty}\beta^{\prime}(q)$.
Then, $0<\alpha_{\min}\le\alpha_{\max}<\infty$.
\item[(iii)]For $\alpha\in\R$, if $\alpha\notin[\alpha_{\min},\alpha_{\max}]$, then $\mcE^{(\mu_f)}(\alpha)=0$, and if $\alpha\in[\alpha_{\min},\alpha_{\max}]$, then 
\[
\mcE^{(\mu_f)}(\alpha)=\inf_{q\in\R}(\beta(q)+q\alpha).
\]
\end{enumerate}
\end{theorem}
For the proof of Theorem \ref{thm:321}, see, e.g., \cite[Theorem A4.2]{Pesin}.


\section{Proof of Theorem \ref{thm:401}}
\label{sec:main}

In this section, we prove Theorem \ref{thm:401}.
In fact, we prove a stronger version of this theorem.
To this end, we extend Definition \ref{def:019} to the case in which $g$ may be defined on another topological Markov shift $(\bTMS,\btms)$. 

\begin{definition}
\label{def:039}
Let $f:\TMS\to\R$ be H\"{o}lder continuous.
Moreover, let $(\bTMS,\btms)$ be another topological Markov shift and $g:\bTMS\to\R$ H\"{o}lder continuous.
We say that the two measure-preserving dynamical systems $(\TMS,\tms,\mu_f)$ and $(\bTMS,\btms,\mu_g)$ are \emph{isomorphic} if there exist two Borel sets $X\subset\TMS,\,Y\subset\bTMS$ and a bijection $\Phi:X\to Y$ such that the following four conditions hold:
\begin{itemize}
\item$\tms(X)\subset X,\ \btms(Y)\subset Y$ and $\mu_f(X)=\mu_g(Y)=1$.
\item Both $\Phi:X\to Y$ and $\Phi^{-1}:Y\to X$ are Borel measurable. 
\item$\Phi\circ\tms=\btms\circ\Phi\ \mbox{on}\ X$.
\item$\mu_f=\mu_g\circ \Phi$.
\end{itemize}
\end{definition}
The map $\Phi:X\to Y$ in Definition \ref{def:039} is called an \emph{isomorphism} from $(\TMS,\tms,\mu_f)$ to $(\bTMS,\btms,\mu_g)$.

The aim of this section is to prove the following theorem, which is a stronger version of Theorem \ref{thm:401}.

\begin{theorem}
\label{thm:405}
Let $f:\TMS\to\R$ be H\"{o}lder continuous.
Moreover, let $(\bTMS,\btms)$ be another topological Markov shift and $g:\bTMS\to\R$ H\"{o}lder continuous.
If $(\TMS,\tms,\mu_f)$ and $(\bTMS,\btms,\mu_g)$ are isomorphic, then $\mcE^{(\mu_f)}=\mcE^{(\mu_g)}$.
\end{theorem}

For the proof of Theorem \ref{thm:405}, we need some lemmas.

Let $f:\TMS\to\R$ be H\"{o}lder continuous and $h_f$ as in Theorem \ref{thm:410}.
We write
\begin{equation}
\label{eq:377}
\widehat{f}=f+\log\frac{h_f}{h_f\circ\tms}.
\end{equation}

Recall from Section \ref{sec:preliminaries} that we write $S_n\phi=\sum_{k=0}^{n-1}\phi\circ\tms^k$.

\begin{lemma}
\label{lem:711}
Let $(\bTMS,\btms)$ be another topological Markov shift and $g:\bTMS\to\R$ H\"{o}lder continuous.
Assume that $(\TMS,\tms,\mu_f)$ and $(\bTMS,\btms,\mu_g)$ are isomorphic. 
For $n\in\N$ and an isomorphism $\Phi:X\to Y$ from $(\TMS,\tms,\mu_f)$ to $(\bTMS,\btms,\mu_g)$, we have
\[
S_{n}\widehat{f}-(S_{n}\widehat{g})\circ\Phi=n\log\lambda_f\lambda_g^{-1},\quad\mu_f\mbox{-a.e.}
\]
\end{lemma}
\begin{proof}
It is easy to see that $\mcL_{\widehat{f}}1=\lambda_f$ and $\int\mcL_{\widehat{f}}\phi\,d\mu_f=\lambda_f\int\phi\,d\mu_f$ for $\phi\in \mcH_{\alpha}(\TMS)$.
The analogues for $g$ also hold.
Hence, we see from Lemma \ref{lem:921} that
\[
\lambda_f^{-n}\exp(S_n\widehat{f})=J_{\mu_f}^{(n)}=J_{\mu_g}^{(n)}\circ\Phi=\lambda_g^{-n}\exp(S_n\widehat{g})\circ\Phi.
\]
Thus, the lemma follows.
\end{proof}

\begin{lemma}
\label{lem:870}
For any $\omega\in\TMS$, we have
\[
P(\tms, f)=\lim_{n\to\infty}\frac{1}{n}\log\left[\sum_{\omega^{\prime}\in\TMS:\,\tms^n\omega^{\prime}=\omega}\exp(S_{n}f(\omega^{\prime}))\right].
\]
\end{lemma}

For the proof of Lemma \ref{lem:870}, see, e.g., \cite[Proposition 4.4.3]{Przytycki}.

We are ready to prove Theorem \ref{thm:405}.

\begin{proof}[Proof of Theorem \ref{thm:405}]
Let $\Phi:X\to Y$ be an isomorphism from $(\TMS,\tms,\mu_f)$ to $(\bTMS,\btms,\mu_g)$.
By Theorem \ref{thm:321}, it is enough to show that
\[
P(\tms,qf)-qP(\tms,f)=P(\btms,qg)-qP(\btms,g),\quad q\in\R.
\]

Without loss of generality, we may assume that $P(\tms,f)=P(\btms,g)=0$.
Let $q\in\R$.
By Lemma \ref{lem:711}, we have $S_{n}\widehat{f}=(S_{n}\widehat{g})\circ\Phi\ \mu_f\mbox{-}a.e.$ for $n\in\N$.
Thus, there exists $\omega\in X$ such that
\begin{equation*}
S_n\widehat{f}(\omega)=S_n\widehat{g}(\Phi\omega)\quad\mbox{and}\quad\tms^{-n}(\{\omega\})\subset X,\quad n\in\N.
\end{equation*}
Let $n\in\N$.
Then, $\Phi(\tms^{-n}(\{\omega\}))\subset\btms^{-n}(\{\Phi(\omega)\})$, and hence,
\begin{align*}
\sum_{\overline{\omega}\in\btms^{-n}(\{\Phi(\omega)\})}e^{qS_ng(\overline{\omega})}\ge&\sum_{\overline{\omega}\in\Phi(\tms^{-n}(\{\omega\}))}e^{qS_ng(\overline{\omega})}\\
=&\sum_{\omega^{\prime}\in\tms^{-n}(\{\omega\})}e^{qS_ng(\Phi(\omega^{\prime}))}=\sum_{\omega^{\prime}\in\tms^{-n}(\{\omega\})}e^{qS_n\widehat{f}(\omega^{\prime})},
\end{align*}
where the first equality follows from the injectivity of $\Phi$.
Thus, by Lemma \ref{lem:870}, we have
\begin{align*}
P(\btms,qg)=&P(\btms,q\widehat{g})\\
=&\lim_{n\to\infty}\frac{1}{n}\log\left[\sum_{\overline{\omega}\in g^{-n}(\{\Phi(\omega)\})}e^{qS_{n}\widehat{g}(\overline{\omega})}\right]\\
\ge&\lim_{n\to\infty}\frac{1}{n}\log\left[\sum_{\omega^{\prime}\in\tms^{-n}(\{\omega\})}e^{qS_n\widehat{f}(\omega^{\prime})}\right]=P(\tms,q\widehat{f})=P(\btms,qf).
\end{align*}
Similarly, we have $P(\btms,qg)\le P(\tms,qf)$.
\end{proof}

\begin{remark}
An analogue of Theorem \ref{thm:405} for two-sided topological Markov shifts and an almost continuous isomorphism which is a hyperbolic structure homomorphism was proved by Schmidt \cite{Schmidt}, Theorem 6.3.
\end{remark}

In the rest of this section, let ${\bf A}$ be a  $2\times 2$ zero-one aperiodic matrix and let $\mcC=L_2(\TMS;\R)$.
Two corollaries to Theorem \ref{thm:405}, i.e., Corollaries \ref{cor:732} and \ref{cor:558} below, give a complete answer to the weak rigidity problem for $\mcC$.
Notice that there exist only three $2\times 2$ zero-one aperiodic matrices
\[
\begin{pmatrix}1&1\\
1&1
\end{pmatrix},\quad\begin{pmatrix}1&1\\
1&0
\end{pmatrix}\quad\mbox{and}\quad\begin{pmatrix}0&1\\
1&1
\end{pmatrix}.
\]

We first consider the case in which ${\bf A}=\left(\begin{smallmatrix}1&1\\
1&1
\end{smallmatrix}\right)$.
For $\alpha\in(0,1)$, we define the probability matrices $P_1(\alpha),P_2(\alpha)$ by
\[
P_1(\alpha)=\begin{pmatrix}1-\alpha&\alpha\\
1-\alpha&\alpha
\end{pmatrix},\quad P_2(\alpha)=\begin{pmatrix}1-\alpha&\alpha\\
\alpha&1-\alpha
\end{pmatrix}.
\]

We set
\begin{align*}
E=\{f\in\mcC:&\ P(f)\neq P_1(\alpha)\ \mbox{and}\ P(f)\neq P_2(\alpha)\\
&\ \mbox{for any}\ \alpha\in(0,1)\ \mbox{with}\ \alpha\neq1/2\}.
\end{align*}
In \cite[Theorem 3]{Barreira}, Barreira and Saraiva proved that
\begin{equation}
\label{eq:182}
\mcE^{(\mu_1(\alpha))}=\mcE^{(\mu_2(\alpha))},\quad\alpha\in(0,1).
\end{equation}
and
\begin{equation}
\label{eq:905}
\{f\in\mcC:[f]_2=[f]_1\}=E.
\end{equation}
Using these results, we prove the following corollary:

\begin{corollary}
\label{cor:732}
Let ${\bf A}=\left(\begin{smallmatrix}1&1\\
1&1
\end{smallmatrix}\right)$ and $\mcC=L_2(\TMS;\R)$.
Then, we have
\[
\{f\in\mcC:[f]_3=[f]_1\}=E.
\]
\end{corollary}
\begin{proof}
Theorem \ref{thm:405} implies that $[f]_2\subset[f]_3\subset[f]_1$ for $f\in\mcC$.
Hence, by (\ref{eq:905}), we have $\{f\in\mcC:[f]_3=[f]_1\}\supset\{f\in\mcC:[f]_2=[f]_1\}=E$.

We prove $\{f\in\mcC:[f]_3=[f]_1\}\subset E$.
We denote by $\mu_1(\alpha)$ and $\mu_2(\alpha)$ the Markov measures associated with $P_1(\alpha)$ and $P_2(\alpha)$, respectively.
Then, from \cite[Theorem 3]{Walters}, $(\TMS,\tms,\mu_1(\alpha))$ and $(\TMS,\tms,\mu_2(\alpha))$ are not isomorphic for any $\alpha\in(0,1)$ with $\alpha\neq1/2$.
Combining this with (\ref{eq:182}), we have $\{f\in\mcC:[f]_3=[f]_1\}\subset E$.
\end{proof}

We next consider the case in which ${\bf A}$ is either $\left(\begin{smallmatrix}1&1\\
1&0
\end{smallmatrix}\right)$ or $\left(\begin{smallmatrix}0&1\\
1&1
\end{smallmatrix}\right)$.
In this case, Barreira and Saraiva showed that
\begin{equation}
\label{eq:334}
\{f\in\mcC:[f]_2=[f]_1\}=\mcC.
\end{equation}
For the proof of (\ref{eq:334}), see \cite[Theorem 4]{Barreira}.
Using this result, we prove the following corollary:

\begin{corollary}
\label{cor:558}
Let ${\bf A}$ be either $\left(\begin{smallmatrix}1&1\\
1&0
\end{smallmatrix}\right)$ or $\left(\begin{smallmatrix}0&1\\
1&1
\end{smallmatrix}\right)$ and let $\mcC=L_2(\TMS;\R)$.
Then, we have
\[
\{f\in\mcC:[f]_3=[f]_1\}=\mcC.
\]
\end{corollary}
\begin{proof}
By the same argument as that in the proof of Corollary \ref{cor:732}, we have $\{f\in\mcC:[f]_2=[f]_1\}\subset\{f\in\mcC:[f]_3=[f]_1\}$.
Thus, by (\ref{eq:334}), we obtain the desired result.
\end{proof}


\appendix
\section{Non-rigidity of entropy spectra}
\label{app:non-rigidity}

We define the function $\delta:\{1,\dots,N\}\to\N\cup\{0\}$ by
\[
\delta(i)=\#\{j\in\{1,\dots,N\}:{\bf A}_{ij}=1\}.
\]
Notice that since ${\bf A}$ is aperiodic, $\delta(i)\ge1$ for any $i\in\{1,\dots,N\}$.
In this appendix, we prove the following theorem:

\begin{theorem}
\label{thm:325}
Let $\mcC=L_2(\TMS;\R)$.
If
\begin{equation}
\label{cond:278}
\#\{i\in\{1,\dots,N\}:\delta(i)=1\}\le1,
\end{equation} 
then the set
\[
\{f\in\mcC:[f]_2=[f]_3\}
\]
contains an open and dense subset of $\mcC$.
\end{theorem}

Recall from Section \ref{sec:preliminaries} that $M$ denotes the set of all $N\times N$ non-negative matrices $A$ such that, for any $i,j\in\{1,\dots,N\}$, $A_{ij}>0$ if and only if ${\bf A}_{ij}=1$.
Also recall that $\lambda(A)$ denotes the Perron root of $A\in M$.
We equip $M$ with the relative topology induced by $\R^{N^2}$.  

For $A\in M$, we denote by $v(A)\in\R^N$ the unique vector such that $v(A)^{\mrmT}$ is a right eigenvector of $\lambda(A)$ with $v(A)_1+\cdots+v(A)_N=1$.
The following well-known lemma is essential to the proof of Theorem \ref{thm:325}.

\begin{lemma}
\label{lem:965}
The following two assertions hold:
\begin{itemize}
\item[(i)]The map $M\ni A\mapsto\lambda(A)\in\R$ is real analytic.
\item[(ii)]For $i\in\{1,\dots,N\}$, the map $M\ni A\mapsto v(A)_i\in\R$ is real analytic.
\end{itemize}
\end{lemma}
\begin{proof}
(i)\
Fix $A_0\in M$ and write $\lambda_0=\lambda(A_0)$.
By \cite[Th\'{e}or\`{e}me 4.3.6]{Chatelin}, there exist a neighbourhood $U\subset M$ of $A_0$ and a positive number $\epsilon$ such that any $A\in U$ has an eigenvalue $\eta(A)\in\C$ satisfying the following two conditions:
\begin{align}
&\eta(A)\ \mbox{is the unique eigenvalue}\ \eta\in\C\ \mbox{of}\ A\ \mbox{such that}\ |\lambda_0-\eta|<\epsilon,\label{cond:390}\\
&|\eta(A)|>|\eta|\ \mbox{holds for any other eigenvalue}\ \eta\in\C\ \mbox{of}\ A.
\label{cond:428}
\end{align}
Notice that (\ref{cond:428}) implies $\eta(A)=\lambda(A)$.

For $A\in M$, we denote by $\Phi_A(\lambda)$ the characteristic polynomial of $A$.
Since the Perron root $\lambda(A)$ has algebraic multiplicity one, the fraction $\Phi_A(\lambda)/(\lambda-\lambda(A))$ is non-zero at $\lambda=\lambda(A)$, and hence, 
\[
\frac{\partial\Phi_A(\lambda)}{\partial\lambda}\bigg|_{\lambda=\lambda_0,\,A=A_0}\neq0.
\]
Thus, by the implicit function theorem, there exist a neighbourhood $V\subset U$ of $A_0$ and a real analytic map $F:V\to\R$ such that, for $A\in V$, $F(A)$ is an eigenvalue of $A$ and $|\lambda_0-F(A)|<\epsilon$. 
It follows from (\ref{cond:390}) that $F(A)=\lambda(A),\ A\in V$.
We conclude that (i) holds. 

(ii)\ 
Let $A\in M$.
Clearly, $v(A)$ is the unique solution of the following system of $N+1$ linear equations in the $N$ variables $x_1,\dots,x_N$:
\begin{equation}
\label{eq:181}
\begin{pmatrix}
A_{11}-\lambda(A)&A_{12}&\cdots&A_{1N}\\
A_{21}&A_{22}-\lambda(A)&\cdots&A_{2N}\\
\vdots&\vdots&\ddots&\vdots\\
A_{N1}&A_{N2}&\cdots&A_{NN}-\lambda(A)\\
1&1&\cdots&1
\end{pmatrix}
\begin{pmatrix}
x_1\\
x_2\\
\vdots\\
x_N
\end{pmatrix}
=\begin{pmatrix}
0\\
0\\
\vdots\\
0\\
1
\end{pmatrix}
\end{equation}
Let $C_A$ be the coefficient matrix of (\ref{eq:181}).
For $i\in\{1,\dots,N\}$, we denote by $C_A^{(i)}$ the $N\times N$ submatrix of $C_A$ obtained by deleting the $i$th row.
Moreover, for $i\in\{1,\dots,N\}$, we denote by $b^{(i)}\in\R^N$ the subvector of $(0,\dots,0,1)\in\R^{N+1}$ obtained by deleting the $i$th coordinate. 

Fix $A_0\in M$. 
Since $C_{A_0}$ has rank $N$, there exist a neighbourhood $U\subset M$ of $A_0$ and a number $i_0\in\{1,\dots,N\}$ such that $C_A^{(i_0)}$ is invertible for any $A\in U$.
Then, for any $A\in U$, $v(A)^{\mrmT}$ is a unique solution of the system of linear equations $C_A^{(i_0)}(x_1,\dots,x_N)^{\mrmT}=(b^{(i_0)})^{\mrmT}$.
Thus, (i) and Cramer's rule imply that, for any $i\in\{1,\dots,N\}$, the map $U\ni A\mapsto v(A)_i\in\R$ is real analytic.
\end{proof}

For $n\in\N$, we set
\begin{align*}
&\widetilde{L}_n:=\{f\in L_n(\TMS;\R):f(\omega)\neq f(\omega^{\prime})\ \mbox{for any}\ \omega,\omega^{\prime}\in\TMS\ \mbox{with}\ \omega|n\neq\omega^{\prime}|n\},\\
&G_n:=\left\{f\in L_n(\TMS;\R):\widehat{f}\in\widetilde{L}_n\right\},
\end{align*}
where $\widehat{f}$ is defined by (\ref{eq:377}).

\begin{lemma}
\label{lem:408}
Let $X,Y\subset\TMS$ and $\Phi:X\to Y$.
Let also $n\in\N$ and $f,g\in L_n(\TMS;\R)$.
Assume that the following four conditions are satisfied:
\begin{itemize}
\item$\tms(X)\subset X$ and $\tms(Y)\subset Y$.
\item Both $X$ and $Y$ are dense in $\TMS$.
\item$\Phi(X)=Y$ and $\Phi\circ\tms=\tms\circ\Phi$ on $X$.
\item$f=g\circ\Phi$ on $X$.
\end{itemize}
Then, the following two assertions hold:
\begin{enumerate}
\item[(i)]$f\in\widetilde{L}_n$ if and only if $g\in\widetilde{L}_n$.
\item[(ii)]Let $f\in\widetilde{L}_n$. 
Then, $\Phi$ has a unique continuous extension $\Phi^*:\TMS\to\TMS$.
Moreover, $\Phi^*\circ\tms=\tms\circ\Phi^*$. 
\end{enumerate}
\end{lemma}
\begin{proof}
(i)\ We have
\begin{equation}
\label{eq:214}
\{f|_{[w]}\}_{w\in W_{{\bf A}}^n}=f(X)=g(\Phi(X))=g(Y)=\{g|_{[w]}\}_{w\in W_{{\bf A}}^n},
\end{equation}
where the first and last equalities follow from the densenesses of $X$ and $Y$, respectively.
It is obvious that $f\in\widetilde{L}_n$ if and only if $\#\{f|_{[w]}\}_{w\in W_{{\bf A}}^n}=\#W_{{\bf A}}^n$.
This and (\ref{eq:214}) imply that (i) holds.

(ii)\ It is enough to show that $\Phi$ is uniformly continuous on $X$.
Thus, we show that the following assertion holds:
\begin{equation*}
\begin{aligned}
\Phi(\omega)|nm=\Phi(\omega^{\prime})|nm\ \mbox{for}\ m\in\N\ \mbox{and}\ \omega,\omega^{\prime}\in X\ \mbox{with}\ \omega|nm=\omega^{\prime}|nm.\\
\end{aligned}
\end{equation*}

First, we consider the case in which $m=1$.
Let $\omega,\omega^{\prime}\in X$ satisfy $\omega|n=\omega^{\prime}|n$.
Then, $g(\Phi(\omega))=f(\omega)=f(\omega^{\prime})=g(\Phi(\omega^{\prime}))$.
Since $g\in\widetilde{L}_n$ from (i), we obtain $\Phi(\omega)|n=\Phi(\omega^{\prime})|n$.

Next, we consider the case in which $m>1$.
Let $\omega,\omega^{\prime}\in X$ satisfy $\omega|nm=\omega^{\prime}|nm$.
It is enough to show that $\tms^{kn}\Phi(\omega)|n=\tms^{kn}\Phi(\omega^{\prime})|n$ for $k\in\{0,\dots,m-1\}$.
Let $k\in\{0,\dots,m-1\}$.
Then, $\tms^{kn}\omega|n=\tms^{kn}\omega^{\prime}|n$, and hence, $\Phi(\tms^{kn}\omega)|n=\Phi(\tms^{kn}\omega^{\prime})|n$.
Since $\Phi$ and $\tms$ commute, we obtain $\tms^{kn}\Phi(\omega)|n=\tms^{kn}\Phi(\omega^{\prime})|n$.
\end{proof}

A homeomorphism $\tau:\TMS\to\TMS$ is called an \emph{automorphism} of $\TMS$ if $\tau\circ\tms=\tms\circ\tau$ holds.
We denote by $\Aut(\TMS)$ the set of all automorphisms of $\TMS$.

\begin{lemma}
\label{lem:087}
Let $n\in\N$ and $f,g\in L_n(\TMS;\R)$.
Assume that $(\TMS,\tms,\mu_f)$ and $(\TMS,\tms,\mu_g)$ are isomorphic and that $f\in G_n$.
Then, for any isomorphism $\Phi:X\to Y$ from $(\TMS,\tms,\mu_f)$ to $(\TMS,\tms,\mu_g)$, there exists a unique continuous extension $\Phi^*:\TMS\to\TMS$ of $\Phi$.
Moreover, $\Phi^*\in\Aut(\TMS)$.
\end{lemma}
\begin{proof}
Remark \ref{rem:785} in Section \ref{sec:preliminaries} implies that $X$ and $Y$ are dense in $\TMS$.
Therefore, by Lemma \ref{lem:711}, we may assume that all of the four conditions in Lemma \ref{lem:408} are satisfied (we replace $f$ and $g$ by $\widehat{f}$ and $\widehat{g}$, respectively).
Since $\widehat{f}\in\widetilde{L}_n$, we see from Lemma \ref{lem:408} (ii) that $\Phi$ has a unique continuous extension $\Phi^*:\TMS\to\TMS$ and $\Phi^*\circ\tms=\tms\circ\Phi^*$.
Moreover, since $\widehat{g}\in\widetilde{L}_n$ from Lemma \ref{lem:408} (i), the inverse $\Phi^{-1}:Y\to X$ also has a unique continuous extension $(\Phi^{-1})^*:\TMS\to\TMS$, and hence, $\Phi^*\in\Aut(\TMS)$ follows.
\end{proof}

\begin{lemma}
\label{lem:891}
Assume that ${\bf A}$ satisfies (\ref{cond:278}).
Then, for any stochastic matrix $P\in M$ and two words $ij,\,kl\in W_{{\bf A}}^2$ with $ij\neq kl$, there exists a stochastic matrix $Q\in M$ such that $P(ij)/P(kl)\neq Q(ij)/Q(kl)$.
\end{lemma}
\begin{proof}
First, we consider the case in which $\delta(i)>1$ and $\delta(k)>1$.
Then, we can take a stochastic matrix $Q\in\ M$ so that $Q(ij)>P(ij)$ and $Q(kl)<P(kl)$.

Next, we consider the case in which $\delta(i)=1$ or $\delta(k)=1$.
Without loss of generality, we may assume that $\delta(k)=1$.
Then, $\delta(i)>1$ since ${\bf A}$ satisfies (\ref{cond:278}).
Thus, we can take a stochastic matrix $Q\in M$ so that $Q(ij)>P(ij)$.
Since $P(kl)=Q(kl)=1$, we obtain $P(ij)/P(kl)<Q(ij)/Q(kl)$.
\end{proof}

\begin{lemma}
\label{lem:226}
Assume that ${\bf A}$ satisfies (\ref{cond:278}).
Then, $G_2$ is an open and dense subset of $L_2(\TMS;\R)$.
\end{lemma}
\begin{proof}
For $ij,kl\in W_{{\bf A}}^2$ with $ij\neq kl$, we define $G(ij;kl)\subset M$ by
\[
G(ij;kl)=\left\{A\in M:\frac{A(ij)}{A(kl)}-\frac{v(A)_i\,v(A)_l}{v(A)_j\,v(A)_k}\neq0\right\}.
\]
Recall the matrix $A(f)$ in (\ref{eq:787}).
The homeomorphism $L_2(\TMS;\R)\ni f\mapsto A(f)\in M$ maps $G_2$ to the intersection $\bigcap_{ij,\,kl\in W_{{\bf A}}^2;\,ij\neq kl}G(ij;kl)$.
Thus, it is enough to prove that $G(ij;kl)$ is an open and dense subset of $M$ for $ij,kl\in W_{{\bf A}}^2$ with $ij\neq kl$.

Fix $ij,kl\in W_{{\bf A}}^2$ with $ij\neq kl$.
From Lemma \ref{lem:965}, the map
\begin{equation}
\label{eq:063}
M\ni A\mapsto \frac{A(ij)}{A(kl)}-\frac{v(A)_i\,v(A)_l}{v(A)_j\,v(A)_k}\in\R
\end{equation}
is real analytic, and hence, continuous.
Thus, $G(ij;kl)$ is open in $M$.

We prove the denseness.
It is easy to see that if $P\in M$ is a stochastic matrix, then $v(P)=(1/N,\dots,1/N)$.
Hence, from Lemma \ref{lem:891}, the map (\ref{eq:063}) is a non-constant real analytic map.
It is well known that the zero set of a non-constant real analytic map has Lebesgue measure zero (see, e.g., \cite{Mityagin}).
Therefore, $G(ij;kl)^c$ has no interior point, and thus, $G(ij;kl)$ is dense in $M$.
\end{proof}

We are ready to prove Theorem \ref{thm:325}.

\begin{proof}[Proof of Theorem \ref{thm:325}]
Lemma \ref{lem:087} implies $G_2\subset\{f\in\mcC:[f]_2=[f]_3\}$.
Moreover, from Lemma \ref{lem:226}, $G_2$ is an open and dense subset of $\mcC$.
\end{proof}

Barreira and Saraiva proved that the set $\{f\in\mcC:[f]_2\subsetneq[f]_1\}$ contains an open and dense subset of $\mcC$ when ${\bf A}=\left(\begin{smallmatrix}
0&1&1\\
1&0&1\\
1&1&0
\end{smallmatrix}\right)$ and $\mcC=L_2(\TMS;\R)$; see \cite[Theorem 5]{Barreira}.
We prove that the set $\{f\in\mcC:[f]_3\subsetneq[f]_1\}$ contains an open and dense subset of $\mcC$, in the same set-up.
Since $[f]_2\subset[f]_3\subset[f]_1$ from Theorem \ref{thm:401}, our result is an extension of the result of Barreira and Saraiva.

\begin{corollary}
\label{cor:778}
Let ${\bf A}=\left(\begin{smallmatrix}
0&1&1\\
1&0&1\\
1&1&0
\end{smallmatrix}\right)$ and $\mcC=L_2(\TMS;\R)$.
Then, the set
\[
\{f\in\mcC:[f]_3\subsetneq[f]_1\}
\]
contains an open and dense subset of $\mcC$.
\end{corollary}
\begin{proof}
By \cite[Theorem 5]{Barreira}, the set $\{f\in\mcC:[f]_2\subsetneq[f]_1\}$ contains an open and dense subset of $\mcC$.
It is clear that ${\bf A}=\left(\begin{smallmatrix}
0&1&1\\
1&0&1\\
1&1&0
\end{smallmatrix}\right)$ satisfies (\ref{cond:278}).
Therefore, the assertion follows from Theorem \ref{thm:325} and the trivial inclusion $\{f\in\mcC:[f]_3\subsetneq[f]_1\}\supset\{f\in\mcC:[f]_2\subsetneq[f]_1\}\cap\{f\in\mcC:[f]_2=[f]_3\}$.
\end{proof}


\section*{Funding}
The author is supported by FY2019 Hiroshima University Grant-in-Aid for Exploratory Research (The researcher support for young scientists).


\end{document}